\documentclass[11pt]{amsart}
\usepackage[utf8]{inputenc}
\usepackage[english]{babel}
\usepackage{amsmath,amsfonts,amssymb}
\usepackage{mathtools}
\usepackage{mathrsfs}
\usepackage[colorlinks=true]{hyperref}
\usepackage{csquotes}
\usepackage{paralist}

\usepackage{tikz}
\usetikzlibrary{cd,arrows}
\usepackage{pgfplots}
\pgfplotsset{width=12cm,compat=1.9}

\usepackage[backend=biber,style=alphabetic,doi=false,isbn=false,url=false]{biblatex}
\theoremstyle{plain}
\newtheorem{thm}{Theorem}[section]
\newtheorem{prp}[thm]{Proposition}

\newtheorem{lem}[thm]{Lemma}
\newtheorem{cnj}[thm]{Conjecture}
\theoremstyle{definition}

\theoremstyle{remark}
\newtheorem{rmk}[thm]{Remark}
\newtheorem{exa}[thm]{Example}
\newtheorem{que}[thm]{Question}

\newcommand{\CC}{\mathbb{C}}
\newcommand{\NN}{\mathbb{N}}
\newcommand{\QQ}{\mathbb{Q}}
\newcommand{\RR}{\mathbb{R}}
\newcommand{\ZZ}{\mathbb{Z}}

\newcommand{\F}{\mathcal{F}}
\newcommand{\N}{\mathcal{N}}
\renewcommand{\O}{\mathcal{O}}

\newcommand{\ol}[1]{\overline{#1}}

\newcommand{\onto}{\twoheadrightarrow}

\DeclarePairedDelimiter\set\{\}
\DeclarePairedDelimiter\abs\lvert\rvert 

\DeclarePairedDelimiter\ideal\langle\rangle

\DeclareMathOperator{\gr}{gr}
\DeclareMathOperator{\Sp}{Sp}

\title[Limit spectral distribution]{Limit spectral distribution\\ for non-degenerate\\ hypersurface singularities}

\author[P.~Almir\'{o}n]{Patricio Almir\'{o}n}
\address{\linebreak
Patricio Almirón\\
Instituto de Matemática Interdisciplinar (IMI), Departamento de \'{A}lgebra, Geometr\'{i}a y Topolog\'{i}a\\
Facultad de Ciencias Matem\'{a}ticas\\
Universidad Complutense de Madrid\\
28040, Madrid, Spain.
}
\email{\href{palmiron@ucm.es}{palmiron@ucm.es}}

\author[M.~Schulze]{Mathias Schulze}
\address{\linebreak
Mathias Schulze\\
Department of Mathematics, TU Kaiserslautern\\
67663 Kaiserslautern\\
Germany
}
\email{\href{mschulze@mathematik.uni-kl.de}{mschulze@mathematik.uni-kl.de}}


\thanks{PA is supported by Spanish Ministerio de Ciencia, Innovaci\'{o}n y Universidades MTM2016-76868-C2-1-P and PID2020-114750GB-C32}

\numberwithin{equation}{section}

\begin{document}
%%%%%%%%%%%%%%%%%%%%%%%%%%%%%%%%%%%%%%%%%%%%%%%%%%%%%%%%%%%%%%%%%%%%%%%%%%%%%%%%

\begin{abstract}
We establish Kyoji Saito's continuous limit distribution for the spectrum of Newton non-degenerate hypersurface singularities.
Investigating Saito's notion of dominant value in the case of irreducible plane curve singularities, we find that the log canonical threshold is strictly bounded below by the doubled inverse of the Milnor number.
We show that this bound is asymptotically sharp.
\end{abstract}

\maketitle
%\tableofcontents

%%%%%%%%%%%%%%%%%%%%%%%%%%%%%%%%%%%%%%%%%%%%%%%%%%%%%%%%%%%%%%%%%%%%%%%%%%%%%%%%
%%%%%%%%%%%%%%%%%%%%%%%%%%%%%%%%%%%%%%%%%%%%%%%%%%%%%%%%%%%%%%%%%%%%%%%%%%%%%%%%
\section{Introduction}
%%%%%%%%%%%%%%%%%%%%%%%%%%%%%%%%%%%%%%%%%%%%%%%%%%%%%%%%%%%%%%%%%%%%%%%%%%%%%%%%
%%%%%%%%%%%%%%%%%%%%%%%%%%%%%%%%%%%%%%%%%%%%%%%%%%%%%%%%%%%%%%%%%%%%%%%%%%%%%%%%

Let $f\colon(\CC^{n+1},0)\to(\CC,0)$ be the germ of a holomorphic function with isolated critical point $0$ and Milnor number $\mu$.
Its spectrum is a discrete invariant formed by $\mu$ rational \emph{spectral numbers} (see \cite[\nopp II.8.1]{Kul98})
\[
\alpha_1,\dots,\alpha_\mu\in\QQ\cap(0,n+1).
\]
They are certain logarithms of the eigenvalues of the monodromy on the middle cohomology of the Milnor fibre which correspond to the equivariant Hodge numbers of Steenbrink's mixed Hodge structure.
In the context of Poincar\'e polynomials it is convenient to consider the spectrum as a polynomial
\[
\Sp_f(t):=\sum_{i=1}^\mu t^{\alpha_i}\in\ZZ[\QQ].
\]

K.~Saito~\cite{Sai81} was the first to study the asymptotic distribution of the spectrum.
He considered the normalized spectrum of $f$,
\[
\chi_f(t):=\frac{\Sp_f(T)}{\mu}=\frac1\mu\sum_{j=1}^\mu T^{\alpha_j},\quad T=\exp(2\pi it),
\]
as the Fourier transform of the discrete probability density on the interval \((0,n+1)\),
\[
\frac{1}{\mu}\sum_{i=1}^{\mu}\delta(s-\alpha_i)ds,
\]
where \(\delta(s)\) is Dirac's delta function.
In the case of Brieskorn--Pham singularities (see Example~\ref{5}) he identified a continuous limit probability distribution $N_{n+1}$ defined by
\[
N_{n+1}(s)ds:=\int_{x_0+\cdots x_n=s}\varphi(x_0)\cdots\varphi(x_n)dx_0\cdots dx_n,
\]
where $\varphi$ is the indicator function of the unit interval $[0,1]$, 
\[
\varphi(x):=
\begin{cases}
1 &\text{if } x\in[0,1],\\
0 &\text{if } x\notin[0,1].
\end{cases}
\]
Under the Fourier transform $\F$, $N_{n+1}$ corresponds to the power
\begin{equation}\label{7}
\F(N_{n+1})=\F(\varphi)^{n+1}.
\end{equation}

%%%%%%%%%%%%%%%%%%%%%%%%%%%%%%%%%%%%%%%%%%%%%%%%%%%%%%%%%%%%%%%%%%%%%%%%%%%%%%%%

K.~Saito~\cite[(2.5)\,i)]{Sai81} suggested to find singularities for which $\chi_f$ converges to $N_{n+1}$.
Our main result establishes his limit spectral distribution for Newton non-degenerate singularities.

%%%%%%%%%%%%%%%%%%%%%%%%%%%%%%%%%%%%%%%%%%%%%%%%%%%%%%%%%%%%%%%%%%%%%%%%%%%%%%%%

\begin{thm}\label{0}
For a fixed Newton diagram $\Gamma$, consider the Newton diagrams $\varpi\Gamma$ obtained from $\Gamma$ by scaling with the factor $\varpi$.
Then we have
\begin{equation}\label{17}
\lim_{\varpi\to\infty}\chi_{f_{\varpi}}=\F(N_{n+1}),
\end{equation}
where $f_{\varpi}$ is any Newton non-degenerate function germ of $n+1$ variables with Newton diagram $\varpi\Gamma$.
\end{thm}

%%%%%%%%%%%%%%%%%%%%%%%%%%%%%%%%%%%%%%%%%%%%%%%%%%%%%%%%%%%%%%%%%%%%%%%%%%%%%%%%

The proof is given in \S\ref{33}.
Combined with the following remark, our result generates new cases where Saito's limit distribution is valid for a suitably chosen limit as in \eqref{17}.
The general choice of limit is unclear.

\begin{rmk}\label{11}
K.~Saito~\cite[(3.7), (3.9), (3.2.1)]{Sai81} proved the following facts.
\begin{enumerate}[(a)]
\item\label{11a} For quasihomogeneous $f$ of degree $1$ with respect to weights $w_0,\dots,w_n$, \eqref{17} holds, even with $\lim_{\varpi\to\infty}$ replaced by $\lim_{w_0,\dots,w_n\to0}$.
\item\label{11b} For irreducible plane curve singularities $f$ with Puiseux pairs $(n_1,l_1),\dots,(n_g,l_g)$,
\[
\lim_{n_g\to\infty}\chi_f=\F(N_2).
\]
\item\label{11c} The join $f+g$ of two functions in disjoint sets of variables satisfies
\[
\chi_{f+g}=\chi_f\cdot\chi_g.
\]
Therefore \eqref{17} is compatible with joins by \eqref{7}.
\end{enumerate}
\end{rmk}

%%%%%%%%%%%%%%%%%%%%%%%%%%%%%%%%%%%%%%%%%%%%%%%%%%%%%%%%%%%%%%%%%%%%%%%%%%%%%%%%

K.~Saito~\cite[(2.5)\,ii), (2.8)\,i)]{Sai81} further suggested to describe up to what extent the spectral distribution is bounded by $N_{n+1}$ and introduced the notion of \emph{(weakly) dominating values}.
Consider the function
\[
\Phi_f\colon[0,1]\to\RR,\quad
r\mapsto\int_{0}^{r}N_{n+1}(s)-\frac1\mu\sum_{i=1}^{\mu}\delta(s-\alpha_i)ds
\]
defined by the difference of the continuous and discrete spectral distributions.
By definition $0<r<\frac{n+1}2$ is a \emph{dominating value} if $\Phi_f(r)>0$ for all $f$ in $n+1$ variables.
A \emph{weakly dominating value} is defined by replacing $<$ by $\leq$ and $\int_{0}^{r}$ by $\int_{0}^{r-\epsilon}$ for all $\epsilon>0$.
In particular, K.~Saito~\cite[(2.8)\,iv)]{Sai81} formulated the following question which by work of M.~Saito~\cite{Sai83} extends Durfee's conjecture on the \emph{geometric genus}
\begin{equation}\label{24}
p_g=\abs{\set{i\mid \alpha_i\le 1}}
\end{equation}
from the case of surface singularities.

%%%%%%%%%%%%%%%%%%%%%%%%%%%%%%%%%%%%%%%%%%%%%%%%%%%%%%%%%%%%%%%%%%%%%%%%%%%%%%%%

\begin{que}\label{21}
Is $1$ a dominating value for all $n\geq 2$? In other words, for $f$ in $n+1$ variables, is the geometric genus bounded by 
\[
p_g<\frac{\mu}{(n+1)!}?
\]
\end{que}

Kerner and Nemethi~\cite{KN17} give a positive answer for Newton non-degenerate singularities with Newton diagram $\varpi\Gamma$ for sufficiently large $\varpi$.

%%%%%%%%%%%%%%%%%%%%%%%%%%%%%%%%%%%%%%%%%%%%%%%%%%%%%%%%%%%%%%%%%%%%%%%%%%%%%%%%

As opposed to Question~\ref{21}, Hertling's variance conjecture \cite[Conj.~6.7]{Her01} addresses the distribution of the spectrum in its full range.

\begin{cnj}[Hertling's Variance Conjecture]\label{20}
The variance of spectral numbers $\alpha_1\le\cdots\le\alpha_\mu$ is bounded by
\[
\frac{1}{\mu}\sum_{i=1}^{\mu}\left(\alpha_i-\frac{n}{2}\right)^2\leq\frac{\alpha_\mu-\alpha_1}{12}.
\]
\end{cnj}

It was confirmed by Br\'elivet~\cite{Bre02,Bre04} for Newton non-degenerate singularities and for plane curves.
We refer to Br\'elivet and Hertling~\cite{HB20} for refined investigations in this direction.

%%%%%%%%%%%%%%%%%%%%%%%%%%%%%%%%%%%%%%%%%%%%%%%%%%%%%%%%%%%%%%%%%%%%%%%%%%%%%%%%

In \S\ref{22}, we investigate (the extremal) spectral numbers below $1$ for their dominance in the case $n=1$ of irreducible plane curve singularities $C=f^{-1}(0)$.
For a single \emph{Puiseux pair} $(p,q)$ we describe these spectral values in terms of the \emph{value semigroup} $S=\ideal{p,q}$ of $C$ (see \eqref{26}).
This can be used to visualize the graph of $\Phi_f$ as a difference (see Figure~\ref{30}).

%%%%%%%%%%%%%%%%%%%%%%%%%%%%%%%%%%%%%%%%%%%%%%%%%%%%%%%%%%%%%%%%%%%%%%%%%%%%%%%%

\begin{figure}[ht]\label{30}
\begin{tikzpicture}[scale=1]
\begin{axis}[axis lines=left,xmajorgrids,ymajorgrids,ytick distance=0.5,xtick distance=1]
\addplot [const plot mark left,color=blue,mark=*,mark size=1pt,fill=white,line width=1pt] 
table [x expr=(\thisrow{X}+14)/45, y expr=\thisrow{Y}/32] {
X Y 
-14 0
0 1
5 2
9 3
10 4
14 5
15 6
18 7
19 8
20 9
23 10
24 11
25 12
27 13
28 14
29 15
30 16
32 17 
33 18
34 19
35 20
36 21
};
\addplot [domain=0:1.15,color=red,line width=1pt] {x^2/2};
\end{axis}
\end{tikzpicture}
\caption{The function $\Phi_f$ as a difference for $S=\ideal{5,9}$.}
\end{figure}

%%%%%%%%%%%%%%%%%%%%%%%%%%%%%%%%%%%%%%%%%%%%%%%%%%%%%%%%%%%%%%%%%%%%%%%%%%%%%%%%

One can write the smallest spectral number, the \emph{log canonical threshold}, as $\frac1p+\frac1q$ and the largest below $1$ as $1-\frac1{pq}$.
For these extremal spectral numbers we prove the following

\begin{prp}\label{18}
If $f(z_0,z_1)$ has a single Puiseux pair $(p,q)$, then
\begin{enumerate}[(a)]
\item\label{18a} $\Phi_f(\frac1p+\frac1q)>0$ unless $p=2$ and $q\in\set{3,5}$, with $\lim_{p\rightarrow \infty}\Phi_f(\frac1p+\frac1q)=0$, 
\item\label{18b} $\Phi_f(1-\frac1{pq})<0$ with $\lim_{p\rightarrow \infty}\Phi_f(1-\frac1{pq})=0$.
\end{enumerate} 
\end{prp}

%%%%%%%%%%%%%%%%%%%%%%%%%%%%%%%%%%%%%%%%%%%%%%%%%%%%%%%%%%%%%%%%%%%%%%%%%%%%%%%%

For general irreducible plane curve singularities, Igusa~\cite[Thm.~1]{Igu77} showed that the log canonical threshold depends only on multiplicity and the first Puiseux exponent (see also \cite[Proof of Thm.~1.1]{Kuw99}). 
It thus equals $\frac1{\ol\beta_0}+\frac1{\ol\beta_1}$ where $\ol\beta_0,\ol\beta_1$ are the two smallest minimal generators of the value semigroup.
The statement of Proposition~\ref{18}.\eqref{18a} remains valid in this extended generality.

\begin{thm}\label{34}
For any irreducible plane curve singularity $C=f^{-1}(0)$ with value semigroup different from $\ideal{2,3}$ and $\ideal{2,5}$, we have $\Phi_f(\frac1{\ol\beta_0}+\frac1{\ol\beta_1})>0$.
In other words, the squared log canonical threshold is bounded by
\[
\left(\frac1{\ol\beta_0}+\frac1{\ol\beta_1}\right)^2>\frac2\mu.
\]
Moreover, $\lim_{n_g\rightarrow \infty}\Phi_f(\frac{1}{\ol\beta_{0}}+\frac{1}{\ol\beta_{1}})=0$.
\end{thm}

In particular, Theorem~\ref{34} provides a quite surprising constraint on the first Puiseux pair of an irreducible plane curve singularity with a given Milnor number.

%%%%%%%%%%%%%%%%%%%%%%%%%%%%%%%%%%%%%%%%%%%%%%%%%%%%%%%%%%%%%%%%%%%%%%%%%%%%%%%%
\subsection*{Acknowledgements}
%%%%%%%%%%%%%%%%%%%%%%%%%%%%%%%%%%%%%%%%%%%%%%%%%%%%%%%%%%%%%%%%%%%%%%%%%%%%%%%%

The first named author wants to thank the second named author for his kindness and facilities for hosting him at TU Kaiserslautern in a pleasant working atmosphere during his research stay in September--November 2020 despite the difficulties of travels and face to face work due to the COVID19 pandemic.

%%%%%%%%%%%%%%%%%%%%%%%%%%%%%%%%%%%%%%%%%%%%%%%%%%%%%%%%%%%%%%%%%%%%%%%%%%%%%%%%
%%%%%%%%%%%%%%%%%%%%%%%%%%%%%%%%%%%%%%%%%%%%%%%%%%%%%%%%%%%%%%%%%%%%%%%%%%%%%%%%
\section{Spectrum of non-degenerate singularities}\label{19}
%%%%%%%%%%%%%%%%%%%%%%%%%%%%%%%%%%%%%%%%%%%%%%%%%%%%%%%%%%%%%%%%%%%%%%%%%%%%%%%%
%%%%%%%%%%%%%%%%%%%%%%%%%%%%%%%%%%%%%%%%%%%%%%%%%%%%%%%%%%%%%%%%%%%%%%%%%%%%%%%%

Suppose that $f\colon(\CC^{n+1},0)\to(\CC,0)$ is \emph{Newton non-degenerate}. 
This means that there are local coordinates $z_0,\dots,z_n$ such that 
\[
f=f(z_0,\dots,z_n)\in\CC\{z_0,\dots,z_n\}=\O_{\CC^{n+1},0}=:\O
\] 
is a \emph{Newton non-degenerate convenient} power series (see \cite[1.19~Def.]{Kou76} and \cite[\nopp II.8.5]{Kul98}).
Let \(\Gamma\) denote the Newton diagram of \(f\). 
We write $\sigma\in\Gamma$ to indicate that $\sigma$ is a face of $\Gamma$.
For $\sigma,\tau\in\Gamma$, write \(\tau\leq\sigma\) if \(\tau\) is a face of \(\sigma\).
By $g_\sigma$ we denote the polynomial obtained from the power series $g\in\O$ by restricting the monomial support to the cone of $\sigma$.

%%%%%%%%%%%%%%%%%%%%%%%%%%%%%%%%%%%%%%%%%%%%%%%%%%%%%%%%%%%%%%%%%%%%%%%%%%%%%%%%

There is a (decreasing) Newton filtration $\N$ defined by $\Gamma$ on $\O$.
Following Steenbrink~\cite[(5.6)]{Ste77}, denote the Newton graded ring associated to $\O$ by 
\[
A:=\gr_\N\O
\]
and the principal parts of \(z_0\frac{\partial f}{\partial z_0},\dots,z_n\frac{\partial f}{\partial z_n}\) with respect to $\N$ by \(F_0,\dots,F_n\).
For $\sigma\in\Gamma$ let $A_\sigma$ be the corresponding graded subring of $A$ and denote by 
\[
d(\sigma):=\dim A_\sigma=\dim\QQ\sigma=\dim\sigma+1
\]
its dimension.

%%%%%%%%%%%%%%%%%%%%%%%%%%%%%%%%%%%%%%%%%%%%%%%%%%%%%%%%%%%%%%%%%%%%%%%%%%%%%%%%

The \emph{Brieskorn module} (see \cite{MP83,Sai88})
\[
\Omega_f:=\Omega^{n+1}_{(\CC^{n+1},0)}/df\wedge\Omega^{n}_{(\CC^{n+1},0)}
\]
also carries a Newton filtration which is induced by the inclusion
\begin{equation}\label{2}
\begin{tikzcd}
\Omega^{n+1}_{(\CC^{n+1},0)}=\O dz_0\wedge\cdots\wedge dz_n\ar[twoheadrightarrow]{d}\ar[hookrightarrow]{r}{\frac{z_0\cdots z_n}{dz_0\wedge\cdots\wedge dz_n}} & \O\ar[twoheadrightarrow]{d}\\
\Omega_f\ar[hookrightarrow]{r} & \O\big/\ideal{z_0\frac{\partial f}{\partial z_0},\dots,z_n\frac{\partial f}{\partial z_n}}.
\end{tikzcd}
\end{equation}
M.~Saito~\cite{Sai88} and Varchenko--Khovanski\u{\i}~\cite{VK85} identified the Poincar\'e series of $\Omega_f$ with the singularity spectrum of $f$ defined by Steenbrink~\cite{Ste77}.

%%%%%%%%%%%%%%%%%%%%%%%%%%%%%%%%%%%%%%%%%%%%%%%%%%%%%%%%%%%%%%%%%%%%%%%%%%%%%%%%

\begin{thm}[M.~Saito, Varchenko--Khovanski\u{\i}]\label{1}
For Newton non-degenerate $f$, the Poincar\'e series of the Newton filtered vector space $\Omega_f$ reads 
\[
p_{\Omega_f}(t)=t^{\alpha_1}+\dots+t^{\alpha_\mu}=:\Sp_f(t)
\]
where $\alpha_1,\dots,\alpha_\mu$ are the spectral numbers of $f$.\qed
\end{thm}

%%%%%%%%%%%%%%%%%%%%%%%%%%%%%%%%%%%%%%%%%%%%%%%%%%%%%%%%%%%%%%%%%%%%%%%%%%%%%%%%

The inclusion~\eqref{2} identifies
\begin{align}\label{6}
\Omega_f&\cong z_0\cdots z_n\O\big/\ideal{z_0\frac{\partial f}{\partial z_0},\dots,z_n\frac{\partial f}{\partial z_n}},\\
\gr_\N\Omega_f&\cong A/\ideal{F_0,\dots,F_n}=:H_f.\nonumber
\end{align}
Based on results of Kouchnirenko~\cite{Kou76} (and Hochster~\cite{Hoc72}) Steenbrink \cite[(5.7)]{Ste77} gave a formula for for Newton non-degenerate $f$ decomposing $p_{H_f}=\Sp_f$ with respect to faces of the Newton diagram:
For a face $\sigma\in\Gamma$ he first writes the Poincar\'e series of the subspace of $A_\sigma/\ideal{F_{0,\sigma},\dots,F_{n,\sigma}}$ corresponding to the interior of the cone $\QQ_{\geq0}\sigma$ of $\sigma$ as
\[
q_{\sigma}(t)=\sum_{\tau\leq\sigma}(-1)^{d(\sigma)-d(\tau)}(1-t)^{d(\tau)}p_{A_{\sigma}}(t).
\]
Denote the minimal dimension of a coordinate space containing $\sigma\in\Gamma$ by
\[
k(\sigma):=\min\set{k\in\ZZ\mid\exists i_1,\dots,i_k\in\set{0,\dots,n}\colon \sigma\subset\QQ e_{i_1}+\cdots+\QQ e_{i_k}}.
\]
Then Steenbrinks formula is given by

%%%%%%%%%%%%%%%%%%%%%%%%%%%%%%%%%%%%%%%%%%%%%%%%%%%%%%%%%%%%%%%%%%%%%%%%%%%%%%%%

\begin{thm}[Steenbrink]\label{3}
For Newton non-degenerate $f$ in $n+1$ variables, the Poincar\'e series of $H_f$ can be written as
\begin{align*}\pushQED{\qed}
p_{H_f}(t)&=\sum_{\sigma\in\Gamma}(-1)^{n+1-d}(\sigma)(1-t)^{k(\sigma)}p_{A_\sigma}(t)\\
&=\sum_{\tau\leq \sigma\in\Gamma}(-1)^{n+1-d(\sigma)}(1-t)^{k(\sigma)-d(\sigma)}q_{\sigma}(t).\qedhere
\end{align*}
\end{thm}

%%%%%%%%%%%%%%%%%%%%%%%%%%%%%%%%%%%%%%%%%%%%%%%%%%%%%%%%%%%%%%%%%%%%%%%%%%%%%%%%
%%%%%%%%%%%%%%%%%%%%%%%%%%%%%%%%%%%%%%%%%%%%%%%%%%%%%%%%%%%%%%%%%%%%%%%%%%%%%%%%
\section{Irreducible plane curve singularities}\label{22}
%%%%%%%%%%%%%%%%%%%%%%%%%%%%%%%%%%%%%%%%%%%%%%%%%%%%%%%%%%%%%%%%%%%%%%%%%%%%%%%%
%%%%%%%%%%%%%%%%%%%%%%%%%%%%%%%%%%%%%%%%%%%%%%%%%%%%%%%%%%%%%%%%%%%%%%%%%%%%%%%%

In this section, we elaborate on the case $n=1$ where $f$ defines an irreducible plane curve singularity $C=f^{-1}(0)$.
We first consider the case of a single Puiseux pair and prove Proposition~\ref{18}, then move on to the general case and prove Theorem~\ref{34}.

Suppose first that $C$ as a single Puiseux pair $(p,q)$.
Then $f$ is Newton non-degenerate with Newton diagram $\Gamma$ consisting of a single line segment $[(p,0),(q,0)]$ and defines an irreducible plane curve singularity $C=f^{-1}(0)$.
The function $f$ is semiquasihomogeneous of weighted degree $1$ with respect to weights
\begin{equation}\label{28}
w_0=\frac1p,\quad w_1=\frac1q,\quad d:=pq,
\end{equation}
on variables $z_0,z_1$ and can be written explicitly as 
\[
f(z_0,z_1)=z_0^p+z_1^q+\sum_{iq+jp>d}a_{i,j}z_0^iz_1^j.
\]
The normalization $\tilde C\onto C$ is given by 
\begin{equation}\label{16}
\begin{tikzcd}
\O_C=\O/\ideal{f}\ar[hookrightarrow]{r} & \O_{\tilde C}=\tilde\O_C\cong\CC\{t\},\\[-20pt]
\ol z_0\ar[mapsto]{r} & t^q+\cdots,\\[-20pt]
\ol z_1\ar[mapsto]{r} & t^p+\cdots.
\end{tikzcd}
\end{equation}
The valuation $\nu\colon\tilde\O_C\to\NN$, $\nu(t)=1$, defines the \emph{value semigroup}
\[
S:=\nu(\O_C\setminus\set0)=\ideal{p,q}\subset\NN.
\]
Due to the finiteness of the normalization $S$ has a finite set of \emph{gaps} $\NN\setminus S$, which yields $k+\NN\subset S$ for $k\gg0$.
The minimal such $k$ is the \emph{conductor} of $S$ and equals the Milnor number (see \cite[Prop.1.2.1.1)]{BG80})
\begin{equation}\label{29}
\mu=(p-1)(q-1).
\end{equation}
The Gorenstein property of $C$ is reflected by the symmetry between elements and gaps (see \cite{Kun70})
\begin{equation}\label{25}
\begin{tikzcd}
S\ar{r}{1:1} & \ZZ\setminus S,\\[-20pt]
a\ar[mapsto]{r} & \mu-1-a.
\end{tikzcd}
\end{equation}
The normalized valuation $\nu/d$ induces the filtration $\O_C$ defined by weights $w=(w_0,w_1)$ on $z_0,z_1$.
By assumption, this is the Newton filtration $\N$.
Factoring \eqref{16} as 
\[
\begin{tikzcd}
\O_C\ar[hookrightarrow]{r}\ar[twoheadrightarrow]{d} & \CC\{t\}\ar[twoheadrightarrow]{d}\\
\O/\ideal{z_0\frac{\partial f}{\partial z_0},z_1\frac{\partial f}{\partial z_1}}\ar{r} & \CC\{t\}/\ideal{t^d}
\end{tikzcd}
\]
and using \eqref{2} yields a Newton filtered inclusion
\[
\begin{tikzcd}
\O/\N_{1+w_0+w_1}\O\ar{r}{dz_0\wedge dz_1}[swap]{\cong} & \Omega_f/\N_{1}\Omega_f\ar[hookrightarrow]{r} & \CC\{t\}/\ideal{t^d}.
\end{tikzcd}
\]
This identifies the corresponding ranges of spectral numbers and of values in the semigroup by means of
\begin{equation}\label{26}
\begin{tikzcd}
\set{\alpha\in\set{\alpha_1,\dots,\alpha_\mu}\mid\alpha<1+w_0+w_1}\ar[leftrightarrow]{r}{1:1} & S/\ideal{d},\\[-20pt]
\alpha\ar[mapsto]{r} & d\alpha-p-q,\\[-20pt]
\frac kd+w_0+w_1 & k\ar[mapsto]{l}.
\end{tikzcd}
\end{equation}
The smallest spectral number $w_0+w_1$ corresponds to $0\in S$, and the gap $\mu-1$ of $S$ defining the Gorenstein symmetry \eqref{25} corresponds to the non-spectral number $1$.
It follows that \eqref{24} can be written explicitly as
\begin{equation}\label{27}
p_g=\abs{\set{(i,j)\in\mathbb{N}^2\mid\frac{i+1}{p}+\frac{j+1}{q}<1}}=\abs{\NN\setminus S}=\frac\mu2.
\end{equation}
Under \eqref{25} the gap $1\in\NN\setminus S$ is the mirror of $\mu-2\in S$ and corresponds to the largest spectral number $1-w_0w_1<1$ by \eqref{26}.

%%%%%%%%%%%%%%%%%%%%%%%%%%%%%%%%%%%%%%%%%%%%%%%%%%%%%%%%%%%%%%%%%%%%%%%%%%%%%%%%

After these preparations we are ready to give the 

\begin{proof}[Proof of Proposition~\ref{18}]\
\begin{asparaenum}[(a)]

\item Using \eqref{28} and \eqref{29} we compute
\begin{align}\label{31}
\mu\Phi_f(w_0+w_1)&=\frac{\mu}{2}(w_0+w_1)^2-1=\frac{(p-1)(q-1)}{2}\left(\frac1p+\frac1q\right)^2-1\\\nonumber
&=\frac{(p-1)(q-1)(p+q)^2-2p^2q^2}{2p^2q^2}\\\nonumber
&=\frac{(pq-p-q+1)(p^2+2pq+q^2)-2p^2q^2}{2p^2q^2}\\\nonumber
&=\frac{2pq+p^3q-3p^2q-p^3+p^2+pq^3-3pq^2-q^3+q^2}{2p^2q^2}\\\nonumber
&=\frac{1}{pq}+\frac{pq-3q-p+1}{2q^2}+\frac{pq-3p-q+1}{2p^2},
\end{align}
which tends to $0$ for $p\to\infty$.
If $p\geq 4$ and $q\geq 5$, then \eqref{31} is positive since
\begin{align*}
pq-3q-p+1&=p(q-1)-3q+1\ge 4q-4-3q+1\geq q-3>0,\\
pq-3p-q+1&=p(q-3)-q+1\ge 4q-12-q+1\geq 3q-11>0.
\end{align*}
If $p=3$, then \eqref{31} becomes
\[
\frac{1}{3q}-\frac{2}{2q^2}+\frac{2q-8}{18}=\frac{2q^3-8q^2+6q-18}{18q^2},
\]
which is positive if $q\geq 4$.
Finally, if $p=2$, then \eqref{31} becomes
\[
\frac{1}{2q}-\frac{q+1}{2q^2}+\frac{q-5}{8}=\frac{q^3-5q^2-4}{8q^2},
\]
which is positive if $q\geq 6$, but negative if $q\in\set{3,5}$.

\item Using \eqref{27} and \eqref{28} we compute
\[
\Phi_f(1-w_0w_1)=\frac12(1-w_0w_1)^2-\frac12=-\frac1{2d^2}=-\frac{1}{2p^2q^2}<0,
\]
which tends to $0$ for $p\to\infty$.

\end{asparaenum}
\end{proof}

%%%%%%%%%%%%%%%%%%%%%%%%%%%%%%%%%%%%%%%%%%%%%%%%%%%%%%%%%%%%%%%%%%%%%%%%%%%%%%%%

Consider now an irreducible plane curve singularity $C=f^{-1}(0)$ with arbitrary number $g$ of Puiseux pairs. 
To prepare the proof of Theorem~\ref{34}, we review some standard integer invariants (see \cite[Ch.~II,\,\S1-3]{Zar06}):
Let
\[
\ol\beta_0<\ol\beta_1<\dots<\ol\beta_g
\]
denote the minimal generators of the value semigroup of $C$ and set
\begin{equation}\label{40}
e_i:=\gcd(\ol\beta_0,\ol\beta_1,\dots,\ol\beta_{i}),\quad
n_i:=\frac{e_{i-1}}{e_i},\quad
q_i:=\frac{\ol\beta_i}{e_{i}}
\end{equation}
for $i=0,\dots,g$.
These greatest common divisors form a strictly decreasing sequence
\begin{equation}\label{36}
\ol\beta_0=e_0>e_1>\cdots>e_g=1.
\end{equation}
Moreover, the minimal generators of the value semigroup satisfy inequalities
\begin{equation}\label{41}
n_{i-1}\ol\beta_{i-1}<\ol\beta_i
\end{equation}
for $i=1,\dots,g$.
The \emph{characteristic Puiseux exponents} of $C$ are defined recursively by
\begin{equation}\label{39}
\beta_0:=\ol\beta_0,\quad
\beta_1:=\ol\beta_1,\quad
\beta_i:=\ol\beta_i-n_{i-1}\ol\beta_{i-1}+\beta_{i-1},
\end{equation}
for $i=2,\dots,g$.
By \eqref{41} they form a strictly increasing sequence
\begin{equation}\label{32}
1\le\beta_0<\beta_1<\cdots<\beta_g.
\end{equation}
The Milnor number of $f$ can be written as (see \cite[Ch.II,\,\S3,\,(3.14)]{Zar06})
\begin{equation}\label{23}
\mu=\sum_{i=1}^g\beta_i(e_{i-1}-e_i)-\beta_0+1.
\end{equation}
On the other hand, A'Campo showed that (see \cite[Thm.~3.(ii)]{ACa73})
\begin{equation}\label{38}
\mu=\sum_{i=1}^ge_i\mu_i,\quad \mu_i:=(n_i-1)(q_i-1).
\end{equation}

%%%%%%%%%%%%%%%%%%%%%%%%%%%%%%%%%%%%%%%%%%%%%%%%%%%%%%%%%%%%%%%%%%%%%%%%%%%%%%%%

\begin{proof}[Proof of Theorem~\ref{34}]
The case where $g=1$ is covered by Proposition~\ref{18}.\eqref{18a}.
Using \eqref{36} and \eqref{32}, we find a lower bound
\begin{align}\label{37}
\mu&=-\beta_ge_g+\sum_{i=1}^{g-1}(\beta_{i+1}-\beta_{i})e_i+\beta_1e_0-\beta_0+1\\\nonumber
&\ge-\beta_ge_g+\beta_1e_0-\beta_0+1\\\nonumber
&=-\beta_g+\beta_0(\beta_1-1)+1\\\nonumber
&>-\beta_g+\beta_0\beta_1-\beta_0.
\end{align}
Suppose first that $g\ge3$.
Using \eqref{23} and \eqref{37}, we compute
\begin{align*}
(\beta_0+\beta_1)^2\mu-2\beta_0^2\beta_1^2
&=(\beta_0^2+2\beta_0\beta_1+\beta_1^2)\mu-2\beta_0^2\beta_1^2\\
&>\sum_{i=1}^g\beta_0^2\beta_i(e_{i-1}-e_i)-\beta_0^3+\beta_0^2\\
&+\sum_{i=1}^g\beta_1^2\beta_i(e_{i-1}-e_i)-\beta_0\beta_1^2+\beta_1^2\\
&-2\beta_0\beta_1\beta_g-2\beta_0^2\beta_1\\
&>\sum_{i=1}^{g-1}\beta_0^2\beta_i
+\sum_{i=1}^{g-1}\beta_1^2\beta_i
-\beta_0^3-\beta_0\beta_1^2\\
&+(\beta_0-\beta_1)^2\beta_g-2\beta_0^2\beta_1\\
&>2\beta_0^2\beta_1+2\beta_1^3-\beta_0^3-\beta_0\beta_1^2-2\beta_0^2\beta_1>0.
\end{align*}
It follows that
\[
\Phi_f\left(\frac1{\ol\beta_0}+\frac1{\ol\beta_1}\right)=
\frac{1}{2}\left(\frac1{\ol\beta_0}+\frac1{\ol\beta_1}\right)^2-\frac1\mu=\frac{(\beta_{0}+\beta_1)^2\mu-2\beta_0^2\beta_1^2}{2 \beta_0^2\beta_1^2\mu}>0.
\]

%%%%%%%%%%%%%%%%%%%%%%%%%%%%%%%%%%%%%%%%%%%%%%%%%%%%%%%%%%%%%%%%%%%%%%%%%%%%%%%%

Suppose now that $g=2$.
By \eqref{38}, \eqref{40}, \eqref{36} and \eqref{41}, 
\[
e_1\mu_1=(n_1-1)(\ol\beta_1-e_1)=n_1\ol\beta_1-\ol\beta_1-e_0 +e_1<n_1\ol\beta_1\le\ol\beta_2-1=q_2-1
\]
and hence
\begin{align}\label{42}
\mu-e_1^2\mu_1&=e_1\mu_1+e_2\mu_2-e_1^2\mu_1\\\nonumber
&>e_1(1+e_2(n_2-1)-e_1)\mu_1\\\nonumber
&=e_1(1+e_1-e_2-e_1)\mu_1\\\nonumber
&=e_1(1-e_2)\mu_1=0.
\end{align}
If $(n_1,q_1)\notin\set{(2,3),(2,5)}$, then by \eqref{40}, Proposition \ref{18} and \eqref{42}
\[
\Phi_f\left(\frac1{\ol\beta_0}+\frac1{\ol\beta_1}\right)=
\frac{1}{e_1^2}\frac{1}{2}\left(\frac{1}{n_1}+\frac{1}{q_1}\right)^2-\frac1\mu>\frac{1}{e_1^2}\frac{1}{\mu_1}-\frac1\mu=\frac{\mu-e_1^2\mu_1}{e_1^2\mu_1\mu}>0.
\]
Otherwise, we have $n_1=2\le e_1=n_2$ and \eqref{41} yields $q_2>2e_1q_1$.
Using \eqref{38} it follows that
\begin{align*}
\mu&=(e_1-1)(q_2-1)+e_1(q_1-1)\\
&>(e_1-1)(2e_1q_1-1)+e_1(q_1-1)\\
&=2e_1q_1(e_1-1)+e_1(q_1-2)+1\\
&=2q_1e_1^2-(q_1+2)e_1+1
\end{align*}
and hence 
\[
(2+q_1)^2\mu-8e_1^2q_1^2=
\begin{cases}
78e^2_1-125e_1+25 & \text{if }q=3,\\
290e^2_1-343e_1+49 & \text{if }q=5.
\end{cases}
\]
In both cases $e_1\ge2$ implies
\[
\Phi_f\left(\frac1{\ol\beta_0}+\frac1{\ol\beta_1}\right)
=\frac{(n_1+q_1)^2\mu-2e_1^2n_1^2q_1^2}{2e_1^2n_1^2q_1^2\mu}
=\frac{(2+q_1)^2\mu-8e_1^2q_1^2}{8e_1^2q_1^2\mu}>0,
\]
which tends to $0$ for $n_g\rightarrow\infty$ since this entails $e_1\rightarrow\infty$ by \eqref{36} and $\mu\rightarrow\infty$ by \eqref{38}.
\end{proof}

%%%%%%%%%%%%%%%%%%%%%%%%%%%%%%%%%%%%%%%%%%%%%%%%%%%%%%%%%%%%%%%%%%%%%%%%%%%%%%%%
%%%%%%%%%%%%%%%%%%%%%%%%%%%%%%%%%%%%%%%%%%%%%%%%%%%%%%%%%%%%%%%%%%%%%%%%%%%%%%%%
\section{Limit spectral distribution}\label{33}
%%%%%%%%%%%%%%%%%%%%%%%%%%%%%%%%%%%%%%%%%%%%%%%%%%%%%%%%%%%%%%%%%%%%%%%%%%%%%%%%
%%%%%%%%%%%%%%%%%%%%%%%%%%%%%%%%%%%%%%%%%%%%%%%%%%%%%%%%%%%%%%%%%%%%%%%%%%%%%%%%

In this section we return to the general setup of \S\ref{19} and prove our main result Theorem~\ref{0}.
Our approach is to subdivide the Newton diagram and mimic an argument of K.~Saito (see \cite[(2.2),(3.7)]{Sai81}).
We begin with his motivating

%%%%%%%%%%%%%%%%%%%%%%%%%%%%%%%%%%%%%%%%%%%%%%%%%%%%%%%%%%%%%%%%%%%%%%%%%%%%%%%%

\begin{exa}[Brieskorn--Pham type singularities]\label{5}
Suppose first that $n=0$ and $f=f(z)=z^d$ is quasihomogeneous of degree $1$ with respect to the weight $w=1/d$ on $z$ with Milnor number $\mu=d-1$.
By \eqref{6}, $H=\ideal{z}\subset\CC\{z\}/\ideal{z^d}$ and hence
\[
p_{H_f}(t)=\frac{t(t^{w-1}-1)}{1-t^w}=\frac{t-t^w}{t^w-1}.
\]
By Theorem~\ref{3}, using L'H\^opital's rule in the second step,
\begin{align}\label{10}
\lim_{w\to0}\chi_f(t)&=\lim_{w\to0}\frac{p_H(T)}{\mu}\\\nonumber
&=\lim_{w\to0}\frac{w}{1-w}\frac{\exp(2\pi it)-\exp(2\pi itw)}{\exp(2\pi itw)-1}\\\nonumber
&=\lim_{w\to0}\frac{\exp(2\pi it)-\exp(2\pi itw)-2\pi itw\exp(2\pi itw)}{1-\exp(2\pi itw)+2\pi it(1-w)\exp(2\pi itw)}\\\nonumber
&=\frac{\exp(2\pi it)-1}{2\pi it}
=\frac{\exp(\pi it)}{\pi t}\frac{\exp(\pi it)-\exp(-\pi it)}{2 i}\\\nonumber
&=\frac{\exp(\pi it)}{\pi t}\sin(\pi t)=\F(\varphi)(t).
\end{align}
Consider now $f=f_0+\cdots+f_n$, where $f_j=f_j(z_j)=z_j^{d_j}$, which is quasihomogeneous of degree $1$ with respect to weights $w_0=1/d_0,\dots,w_n=1/d_n$ on the variables $z_0,\dots,z_n$ with Milnor number $\mu=\mu_f=\prod_{j=0}^n\mu_{f_j}$.
Then $H=H_f=H_{f_0}\otimes_\CC\cdots\otimes_\CC H_{f_n}$ and hence, by the first part and \eqref{7}, 
\[
\lim_{w_0,\dots,w_n\to0}\chi_f(t)=\prod_{i=0}^n\lim_{w_i\to0}\chi_{f_i}(t)=\F(\varphi)^{n+1}(t)=\F(N_{n+1})(t).
\] 
In this sense the normalized spectrum converges in distribution to the continuous probability distribution $N_{n+1}$.
\end{exa}

%%%%%%%%%%%%%%%%%%%%%%%%%%%%%%%%%%%%%%%%%%%%%%%%%%%%%%%%%%%%%%%%%%%%%%%%%%%%%%%%

For our purpose we adapt the calculation \eqref{10} as follows.

\begin{lem}\label{9}
$\lim_{w\to0}w\frac{1-T}{1-T^w}=\F(\varphi)(t)$.
\end{lem}
\begin{proof}
Using L'H\^opital's rule in the second step and \eqref{10}, we compute
\begin{align*}
\lim_{w\to0}w\frac{1-T}{1-T^w}
&=\lim_{w\to0}\frac{w\cdot(1-\exp(2\pi it))}{1-\exp(2\pi itw)}\\
&=\lim_{w\to0}\frac{1-\exp(2\pi it)}{-2\pi it\exp(2\pi itw)}\\
&=\frac{1-\exp(2\pi it)}{-2\pi it}=\F(\varphi)(t).\qedhere
\end{align*}
\end{proof}

%%%%%%%%%%%%%%%%%%%%%%%%%%%%%%%%%%%%%%%%%%%%%%%%%%%%%%%%%%%%%%%%%%%%%%%%%%%%%%%%

For the subdivision of the Newton diagram we rely on the following general result.
The basis of a rational pointed cone $\sigma$ are the irreducible integral vectors $\alpha_0,\dots,\alpha_k$ on its rays.
If it extends to a lattice basis, then $\sigma$ is called \emph{regular}.
In this case $\sigma$ is a simplicial cone and the convex hull of $\set{0,\alpha_0,\dots,\alpha_k}$ has $k$-dimensional volume $1$ (see \cite[\S1.1]{Kho83}).
A rational fan is called \emph{regular} if all its cones are regular. 
Varchenko~\cite[Thm.~1, Remark]{Kho83} pointed out the following 

\begin{thm}\label{8}
Any finite rational fan has a regular subdivision.\qed
\end{thm}

%%%%%%%%%%%%%%%%%%%%%%%%%%%%%%%%%%%%%%%%%%%%%%%%%%%%%%%%%%%%%%%%%%%%%%%%%%%%%%%%

Finally, we make use of Kouchnirenko's formula~\cite[Thm. I]{Kou76} for the Milnor number of Newton non-degenerate singularities in terms of volumes.

\begin{thm}[Kouchnirenko]\label{4}
The Milnor number of any Newton non-degenerate $f$ in $n+1$ variables can be written as
\[
\mu=(n+1)!V_{n+1}-n!V_{n}+\cdots+1!(-1)^{n}V_1+(-1)^{n+1},
\]
where $V_k$ is the sum of $k$-dimensional volumes of the intersection of the convex hull of $\Gamma\cup\set0$ with the $k$-dimensional coordinate planes.\qed
\end{thm} 

%%%%%%%%%%%%%%%%%%%%%%%%%%%%%%%%%%%%%%%%%%%%%%%%%%%%%%%%%%%%%%%%%%%%%%%%%%%%%%%%

We are now ready for the 

\begin{proof}[Proof of Theorem~\ref{0}]
By Theorem~\ref{8}, $\Gamma$ has a subdivision $\tilde\Gamma$ corresponding to a regular subdivision of its fan of cones.
For any $\tau\in\tilde\Gamma$ let $w^\tau_0,\dots,w^\tau_k$ be the weights of the basis of $\QQ_{\geq 0}\tau\cap \ZZ$.
Then 
\[
p_{A_\sigma}(t)=\sum_{\tilde\Gamma\ni\tau\le\sigma}(-1)^{d(\sigma)-d(\tau)}p_{A_\tau}(t),\quad
p_{A_{\tau}}(t)=\prod_{j=0}^{\dim\tau}\frac1{(1-t^{w^\tau_j})}.
\]
Substituting into Steenbrink's formula from Theorem~\ref{3} yields
\[
p_{H_f}(t)=\sum_{\tilde\Gamma\ni\tau\le\sigma\in\Gamma}(-1)^{n+1-d(\tau)}\frac{(1-t)^{k(\sigma)}}{\prod_{j=0}^{\dim\tau}(1-t^{w^\tau_j})}.
\]
Passing to $\varpi\Gamma$, $w_j^\tau$ is preplaced by $\varepsilon w_j^\tau$ where $\varepsilon\varpi=1$ and hence
\[
p_{H_{f_\varpi}}(t)=\sum_{\tilde\Gamma\ni\tau\le\sigma\in\Gamma}(-1)^{n+1-d(\tau)}\frac{(1-t)^{k(\sigma)}}{\prod_{j=0}^{\dim\tau}(1-t^{\varepsilon w^\tau_j})}.
\]
By Theorems~\ref{1} and \ref{4}, 
\begin{align}
\label{12}\lim_{\varpi\to\infty}\chi_{f_\varpi}(t)&=\lim_{\varpi\to\infty}\frac{p_{H_{f_\omega}}(T)}{\mu_{f_\varpi}}\\\nonumber
&=\sum_{\tilde\Gamma\ni\tau\le\sigma\in\Gamma}(-1)^{n+1-d(\tau)}\lim_{\varpi\to\infty}\frac{1}{\mu_{f_\varpi}}\frac{(1-t)^{k(\sigma)}}{\prod_{j=0}^{\dim\tau}(1-t^{\varepsilon w^\tau_j})},\\
\label{13}\mu_{f_\varpi}&=\sum_{j=0}^{n+1}(-1)^{n+1-j}j!\varpi^jV_j.
\end{align}
Fix $\tilde\Gamma\ni\tau\le\sigma\in\Gamma$.
Let $V(\tau)$ be the $d(\tau)$-dimensional volume of the convex hull of $\tau\cup\set0$.
Note that 
\begin{align}
\label{14}\sum_{\substack{\tau\in\tilde\Gamma\\d(\tau)=n+1}}V(\tau)&=V_{n+1},\\
\label{15}1/V(\tau)&=d(\tau)!\prod_{j=0}^{d(\tau)}w^\tau_j.
\end{align}
The summand in \eqref{12} indexed by $\tau$ is then computed using \eqref{15}, Lemma~\ref{9} and \eqref{13}:
\begin{align*}
&\phantom{=}\lim_{\varpi\to\infty}\frac1{\mu_{f_\varpi}}\frac{(1-T)^{k(\sigma)}}{\prod_{j=0}^{\dim\tau}(1-T^{\varepsilon w^\tau_j})}\\
&=\lim_{\varpi\to\infty}\frac{\prod_{j=0}^{\dim\tau}\frac{\varpi}{w^\tau_j}}{\mu_{f_\varpi}}
(1-T)^{k(\sigma)-d(\tau)}
\prod_{j=0}^{\dim\tau}\lim_{\varepsilon\to0}\varepsilon w^\tau_j\frac{1-T}{1-T^{\varepsilon w^\tau_j}}\\
&=\lim_{\varpi\to\infty}\frac{d(\tau)!V(\tau)\varpi^{d(\tau)}}{\mu_{f_\varpi}}
(1-T)^{k(\sigma)-d(\tau)}
\left(\lim_{w\to0}w\frac{1-T}{1-T^w}\right)^{d(\tau)}\\
&=\lim_{\varpi\to\infty}\frac{d(\tau)!V(\tau)\varpi^{d(\tau)}}{\mu_{f_\varpi}}
(1-T)^{k(\sigma)-d(\tau)}
\F(N_{d(\tau)})(t)\\
&=\begin{cases}
\frac{V(\tau)}{V_{n+1}}\F(N_{n+1})(t) & \text{if }d(\tau)=n+1,\\
0 & \text{if }d(\tau)<n+1.
\end{cases}
\end{align*}
The claim now follows by substituting into \eqref{12} and applying \eqref{14}.
\end{proof}
%%%%%%%%%%%%%%%%%%%%%%%%%%%%%%%%%%%%%%%%%%%%%%%%%%%%%%%%%%%%%%%%%%%%%%%%%%%%%%%%
\printbibliography
%%%%%%%%%%%%%%%%%%%%%%%%%%%%%%%%%%%%%%%%%%%%%%%%%%%%%%%%%%%%%%%%%%%%%%%%%%%%%%%%
\end{document}